\documentclass[a4paper, twoside, 11pt]{article}

\usepackage{geometry}

\usepackage[english]{babel}
\usepackage[utf8]{inputenc}

\usepackage{mathpazo}
\usepackage{authblk}
\usepackage{mathrsfs}
\usepackage{amsfonts, amssymb, amsmath, amsthm}
\usepackage{amsbsy}

\usepackage[backref=page]{hyperref}

\usepackage{tabularx}

\usepackage{fancyhdr}
\usepackage{ifthen}
\newtheoremstyle{mystyle}{}{}{\rmfamily}%
{}{\normalfont\bfseries}{ }{ }{}

\hypersetup{colorlinks=true, linkcolor=blue, citecolor=blue, urlcolor=blue}

\theoremstyle{mystyle}
\newtheorem{theorem}{Theorem}[section]
\newtheorem{lemma}[theorem]{Lemma}
\newtheorem{definition}[theorem]{Definition}

\newtheorem{corollary}[theorem]{Corollary}
\newtheorem{proposition}[theorem]{Proposition}
\newtheorem{remark}[theorem]{Remark}

\newtheorem{example}[theorem]{\bf Example}

\newcommand{\keywords}[1]{\par\textbf{Keywords:} #1\par}
\newcommand{\dedicatory}[1]{%
	\thispagestyle{empty}
	\begin{center}
		\textit{#1} 
	\end{center}
	\setcounter{page}{1} 
}

\title{
Super-Shadowing and Supercyclicity
}
\author[1]{Eric Cabezas}
\author[2]{Manuel Saavedra}
\affil[1,2]{Instituto de Matemática, Universidade Federal do Rio de Janeiro, RJ, Brazil.}
\affil[1]{\texttt{eric@im.ufrj.br}}
\affil[2]{\texttt{saavmath@im.ufrj.br}}

\date{}

\fancypagestyle{customstyle}{
	\fancyhf{}
	\fancyhead[LE,RO]{\thepage}
	\fancyhead[CE,CO]{\ifthenelse{\isodd{\thepage}}{Eric cabezas and Manuel Saavedra}{Super-Shadowing and Supercyclicity}}
	
}

\begin{document}
	\maketitle
	
	\setlength{\headheight}{14.49998pt}
	
\dedicatory{This paper is dedicated to Professor Alexander Arbieto, in celebration of the 20th anniversary of the Seminário de Sistemas Dinâmicos}
	
\begin{abstract}
We introduce the super-shadowing property in linear dynamics, where pseudotrajectories are approximated by sequences of the form $(\lambda_n T^n x)$, with $(\lambda_n)_n$ being complex scalars. For compact operators on Banach spaces, we characterize the operators that possess the positive super-shadowing property and the positive limit super-shadowing property. Additionally, we demonstrate that no surjective isometric operator on a separable Banach space \(X\) with \(\text{dim}(X)>1\) can exhibit the positive super-shadowing property. Finally,  we provide some results on upper frequently supercyclic and reiteratively supercyclic operators.
\end{abstract}
	
	
\section{Introduction}

Consider a metric space $X$ equipped with a metric $d$, and let $T: X \to X$ be a homeomorphism. For a given $\delta > 0$, a sequence $(x_n)_{n \in \mathbb{Z}} \subset X$ is called a $\delta$-pseudotrajectory of $T$ if 
\[
d(T(x_n), x_{n+1}) \leq \delta, \quad \forall n \in \mathbb{Z}.
\]

A central concept in the study of pseudotrajectories is the \emph{shadowing property}, first introduced by Bowen \cite{bowen1975omega} and Sinai \cite{sinai1972gibbs} in the context of compact metric spaces. We say that $T$ has the shadowing property if, for every $\epsilon > 0$, there exists $\delta > 0$ such that any $\delta$-pseudotrajectory $(x_n)_{n}$ is $\epsilon$-shadowed by a true trajectory of $T$. This means there exists a point $q \in X$ satisfying 
\begin{equation}\label{defi-shado}
	d(x_n, T^n(q)) < \epsilon, \quad \forall n \in \mathbb{Z}.
\end{equation}

We say that $T$ has the \emph{limit shadowing property} if, for every sequence $(x_n)_{n \in \mathbb{Z}}$ in $X$ satisfying 
\[
\lim_{|n| \to \infty} d(T(x_n), x_{n+1}) = 0,
\]
there exists $q \in X$ such that 
\begin{equation}\label{def-limshado}
	\lim_{|n| \to \infty} d(x_n, T^n(q)) = 0.
\end{equation}

When $T$ is a continuous map, we can define the notion of \emph{positive shadowing} by replacing $\mathbb{Z}$ with $\mathbb{N}_0 := \mathbb{N} \cup \{0\}$ in (\ref{defi-shado}). Similarly, the \emph{positive limit shadowing property} is defined by considering $n \to \infty$ in (\ref{def-limshado}).

For an excellent exposition of the shadowing property and its applications, we refer the reader to \cite{palmer2000shadowing, pilyugin2006shadowing}.

In the context of linear dynamics, several studies have explored the shadowing property and its variations \cite{ antunes2022chain,bernardes2021shadowing,bernardes2024shadowing, bernardes2025generalized, d2021generalized,maiuriello2022expansivity, ombach1994shadowing}. These works encompass a range of other concepts, including hyperbolicity, expansiveness, structural stability, and chain recurrence.

Under the shadowing property, every pseudotrajectory is strictly constrained to approximate the orbit of some vector. Our shift in perspective: we introduce the super-shadowing property as follows. Let \(X\) be a complex Banach space and \(T:X\rightarrow X\) be a continuous linear operator. For each \( \epsilon > 0 \), there exists \( \delta > 0 \) such that for each \( \delta \)-pseudotrajectory, there exists \( p \in X \) and a sequence of complex numbers \( (\lambda_n)_{n} \) such that \( \| x_n - \lambda_n T^n p \| < \epsilon \). In other words, the pseudotrajectories are approximated by rescalings of the elements of the orbit of some vector. Clearly, this is motivated by the notion of supercyclic operators.

We can distinguish two important notions in linear dynamics: hypercyclicity and supercyclicity. We say that an operator $T: X \to X$ is \emph{hypercyclic} if there exists a vector $x \in X$ such that the $T$-orbit of $x$, given by $\{T^n x : n \in \mathbb{N}_0\}$, is dense in $X$. The set of such vectors is denoted by $\text{HC}(T)$. On the other hand, we say that $T: X \to X$ is a \emph{supercyclic operator} if there exists a vector $x \in X$ such that 
\begin{align*}
\{\lambda \cdot T^n x : n \in \mathbb{N}_0, \, \lambda \in \mathbb{C}\}
\end{align*}
is dense in $X$. The set of such vectors is denoted by $\text{SC}(T)$. In both cases, the separability of $X$ is required. For further details, we refer the reader to \cite{bayart2009dynamics, grosse2011linear}.

Bayart and Grivaux introduced the notion of a \emph{frequently hypercyclic operator} \cite{bayart2004hypercyclicite, bayart2006frequently}. More precisely, we say that $T$ is frequently hypercyclic if there exists a vector $x \in X$ such that, for any non-empty open set $U \subset X$,
\begin{align*}
	\underline{\text{dens}}\{n \geq 0 : T^n x \in U\} > 0.
\end{align*}
Such a vector $x$ is called a \emph{frequently hypercyclic vector} for $T$, and the set of all such vectors is denoted by $\text{FHC}(T)$.

A weaker notion was introduced by Shkarin \cite{shkarin2009spectrum}, called an \emph{upper frequently hypercyclic operator}. We say that $T$ is upper frequently hypercyclic if there exists a vector $x \in X$ such that, for any non-empty open set $U \subset X$,
\begin{align*}
	\overline{\text{dens}}\{n \geq 0 : T^n x \in U\} > 0.
\end{align*}
Such a vector $x$ is called an \emph{upper frequently hypercyclic vector} for $T$, and the set of all such vectors is denoted by $\text{UFHC}(T)$.

Peris \cite{peris2005topologically} introduced the notion of \emph{reiterative hypercyclicity}. We say that $T$ is reiteratively hypercyclic if there exists a vector $x \in X$ such that, for any non-empty open set $U \subset X$,
\[
\overline{\text{Bd}}\{n \geq 0 : T^n x \in U\} > 0.
\]
The point $x$ is then called a \emph{reiteratively hypercyclic point} for $T$, and the set of all such points is denoted by $\text{RHC}(T)$.

Two important results concerning the topological nature of $\text{FHC}$, $\text{UFHC}$, and $\text{RHC}$ are as follows: The first of these sets, $\text{FHC}$, is meager \cite{moothathu2013two}, whereas the latter two, $\text{UFHC}$ and $\text{RHC}$, are either empty or residual \cite{bayart2015difference, bes2016recurrence, bonilla2018upper}. The approach taken in \cite{bonilla2018upper} establishes a Birkhoff-type theorem for upper Furstenberg families.

Recently, Alves, Botelho, and Fávaro introduced the notion of \emph{frequently supercyclic operators} \cite{alves2024frequently}. The key difference from hypercyclicity lies in the perspective: for $x \in X$ and a non-empty open set $U \subset X$, they consider the set
\begin{align*}
	\{n \geq 0 : \mathbb{C} \cdot T^n x \cap U \neq \emptyset\}.
\end{align*}

The paper is structured as follows. In Section \ref{sec2}, we present some basic results regarding the super-shadowing property. In Section 3, we characterize, for compact operators on a complex Banach space, the positive super-shadowing property (see Theorem \ref{compact-1}) and the limit super-shadowing property (see Theorem \ref{compact-2}). In the final section, we show that no surjective isometric operator on a Banach space $X$ with $\dim(X) > 1$ possesses the positive super-shadowing property (see Theorem \ref{iso-no-super}). Additionally, we establish a Birkhoff-type theorem for upper Furstenberg families in the context of supercyclic operators (see Theorem \ref{uff-super}).

\section{The Super-Shadowing Property}\label{sec2}

\begin{definition}
	Let \(T\) be an invertible continuous linear operator on a Banach space \(X\). We say that \(T\) possesses the \emph{super-shadowing property} if, for every \(\epsilon > 0\), there exists \(\delta > 0\) such that for any \(\delta\)-pseudotrajectory \((x_n)_{n \in \mathbb{Z}}\), there exists \(q \in X\) satisfying
	\begin{equation}\label{def.ss}
		\|x_n - \lambda_n T^n q\| < \epsilon \quad \text{for all } n\in \mathbb{Z},
	\end{equation}
	where \((\lambda_n)_{n \in \mathbb{Z}}\) is a sequence of nonzero scalars. 
\end{definition}

In the classical shadowing property, the sequence \((\lambda_n)\) can be chosen as constants, specifically \(\lambda_n = 1\) for all \(n\). By rescaling, it is also possible to select \((\lambda_n)\) as a constant sequence \(\lambda_n = \lambda \neq 0\) for all \(n\).

A natural question arises: Is the existence of a specific sequence \((\lambda_n)_{n \in \mathbb{N}}\), for which given any \(\epsilon > 0\), there exists \(\delta > 0\) such that for every  \(\delta\)- pseudotrajectory \((x_n)_{n \in \mathbb{N}}\), there exists \(p \in X\) satisfying
\[
\|x_n - \lambda_n T^n p\| < \epsilon \quad \text{for all } n,
\]
related to the shadowing property? Indeed, this condition turns out to be equivalent to the classical shadowing property. Let us explore this equivalence in detail.

For \(p\in X\setminus \{0\}\), it is clear that \(T^{n}p\neq 0\) for every \(n\in \mathbb{Z}\). One can notice that the sequence \((T^n p)_{n \in \mathbb{Z}}\) forms a \(\delta\)-pseudotrajectory for any \(\delta > 0\), we can, for each \(\epsilon_\ell = \frac{1}{2^\ell}\), find \(q_\ell \in X\) such that
\[
\|T^n p - \lambda_n T^n q_\ell\| < 2^{-\ell}.
\]
For \(n = 0\), it follows immediately that \(\lambda_0 \neq 0\), and furthermore, \(q_\ell\) converges to \(\frac{p}{\lambda_0}\) as \(\ell \to \infty\). A similar argument shows that \(\lambda_n = \lambda_0\) for all \(n\). Consequently, \(T\) satisfies the shadowing property.

\begin{remark}
	To verify whether an operator \(T\) satisfies the super-shadowing property, it suffices to show that for some \(\epsilon > 0\), there exists \(\delta > 0\) such that for every \(\delta\)-pseudotrajectory \((x_n)_{n}\), there exists \(q \in X\) satisfying condition \eqref{def.ss}. Without loss of generality, we can assume \(\lambda_0 = 1\).
\end{remark}

The shadowing property and limit shadowing have been extensively studied in the literature (see, e.g., \cite{bernardes2018expansivity, pilyugin2006shadowing}). For instance, an invertible operator \(T\) on a Banach space \(X\) has the shadowing property if and only if there exists a constant \(K > 0\) such that for every bounded sequence \((z_n)_{n \in \mathbb{Z}}\) in \(X\), there exists a sequence \((x_n)_{n \in \mathbb{Z}}\) in \(X\) satisfying
\[
\sup_{n \in \mathbb{Z}} \|y_n\| \leq K \sup_{n \in \mathbb{Z}} \|z_n\| \quad \text{and} \quad y_{n+1} = T y_n + z_n, \, \forall n \in \mathbb{Z}.
\]

This result motivates us to establish a similar characterization for the super-shadowing property. To this end, observe that if \((x_n)_{n \in \mathbb{N}}\) is a \(\delta\)- pseudotrajectory, then for any \(p \in X\), the sequence \((x_n - T^n p)_{n \in \mathbb{N}}\) is also a \(\delta\)-pseudotrajectory. If, for some \(p \in X\), we can approximate \((x_n - T^n p)_{n \in \mathbb{N}}\), we say that \(T\) possesses the \emph{weak super-shadowing property}. Let us formalize this notion.

\begin{definition}
	Let \(X\) be a Banach space, and let \(T\) be a continuous linear operator on \(X\). We say that \(T\) possesses the \emph{weak super-shadowing property} if, for every \(\epsilon > 0\), there exists \(\delta > 0\) such that for any \(\delta\)-pseudotrajectory \((x_n)_{n}\), there exist \(p, q \in X\) and a sequence of scalars \((\lambda_n)_{n}\) satisfying
	\[
	\|x_n - T^n p - \lambda_n T^n q\| < \epsilon \quad \text{for all } n.
	\]
	When \(n \in \mathbb{N}_0\), we say that \(T\) has the \emph{positive weak super-shadowing property}. For \(n \in \mathbb{Z}\), we say that \(T\) has the \emph{weak super-shadowing property}.
\end{definition}

\begin{lemma}\label{equiv-wsupershado}
	Let \(T\) be an invertible operator on a Banach space \(X\). Then \(T\) possesses the weak super-shadowing property if and only if there exists a constant \(K > 0\) such that for every bounded sequence \((z_{n})_{n}\) in \(X\), there exist a sequence \((y_{n})_{n} \subset X\), a vector \(q \in X\), and a sequence of complex numbers \((\beta_{n})_{n}\) satisfying:
	\begin{align}\label{ineq-supershado}
		\sup_{n \in \mathbb{Z}} \|y_{n}\| \leq K \sup_{n \in \mathbb{Z}} \|z_{n}\|,\;\text{and}\; y_{n+1} = Ty_{n} + z_{n} + \beta_{n} T^{n+1} q,\, \forall n \in \mathbb{Z}.
	\end{align}
\end{lemma}

The proof of this result follows a similar line of argumentation as presented in \cite[Lemma 10]{bernardes2018expansivity}.
\begin{proof}[Proof of Lemma \ref{equiv-wsupershado}]
	Suppose that \(T\) satisfies the weak super-shadowing property. Let \(\delta > 0\) be the constant associated with the definition of weak super-shadowing for \(\epsilon = 1\). Consider a bounded sequence \((z_n)_{n \in \mathbb{Z}}\) with \(M = \sup_{n \in \mathbb{Z}} \|z_n\|\). Fix \(x_0 \in X\) and define the sequence \((x_n)_{n \in \mathbb{Z}}\) recursively as 
	\[
	x_{n+1} = Tx_n + \frac{\delta}{M} z_n \quad \text{for all } n \in \mathbb{Z}.
	\] 
	Observe that \((x_n)_{n \in \mathbb{Z}}\) forms a \(\delta\)-pseudotrajectory of \(T\). By the weak super-shadowing property, there exist \(p, q \in X\) and a sequence of complex numbers \((\lambda_n)_{n \in \mathbb{Z}}\) such that $\|x_n - T^n p - \lambda_n T^n q\| < 1$ for all $n\in \mathbb{Z}$.
	
	To verify \eqref{ineq-supershado}, define \(y_n = \frac{M}{\delta} (x_n - T^n p - \lambda_n T^n q)\), $K=\frac{1}{\delta}$ and consider the sequence  \(\beta_n = \frac{M}{\delta}( \lambda_n-\lambda_{n+1})\). This establishes the forward implication.
	
	Conversely, suppose that the condition \eqref{ineq-supershado} holds. Let \(\epsilon = 1\) and set \(\delta = \frac{1}{2K}\). Consider any \(\delta\)-pseudotrajectory \((x_n)_{n \in \mathbb{Z}}\) for \(T\), and define the sequence \((z_n)_{n \in \mathbb{Z}} \subset X\) by \(z_n := x_{n+1} - Tx_n\), which is a bounded sequence. By assumption, there exist a sequence \((y_n)_{n \in \mathbb{Z}} \subset X\), a vector \(q \in X\), and a sequence of complex numbers \((\beta_n)_{n \in \mathbb{Z}}\) such that $\sup_n \|y_n\| \leq K \sup_n \|z_n\| < 1$ and the recursive relation holds, $y_{n+1} = Ty_n + z_n + \beta_n T^{n+1} q$ for all $n\in \mathbb{Z}$.
	
	As a consequence, 
	\begin{align*}
		y_{n+1} - x_{n+1} = T(y_n - x_n) + \beta_n T^{n+1} q,\, \forall n \in \mathbb{Z}.
	\end{align*}
	In other words, there exists a sequence of complex numbers \((\lambda_n)_{n}\) such that
	\[
	y_n - x_n = T^n (y_0 - x_0) - \lambda_n T^{n}q  \quad \text{for all } n \in \mathbb{Z}.
	\]
	Let \(p := x_0 - y_0\). It follows that $\|x_n - T^n p - \lambda_n T^n q\| = \|y_n\| < 1$.
\end{proof}

\begin{definition}
	Let $T$ be an invertible operator on a Banach space $X$. We say that $T$ has the \emph{limit super-shadowing property} if, for every sequence $(x_n)_{n \in \mathbb{Z}} \subset X$ satisfying  \(	\lim_{|n| \to \infty} \|T x_n - x_{n+1}\| = 0\), there exist $q \in X$ and a sequence of scalars $(\lambda_n)_{n \in \mathbb{Z}}$ such that \(	\lim_{|n| \to \infty} \|x_n - \lambda_n T^n q\| = 0\). If, instead, there exist $p, q \in X$ and a sequence of scalars $(\lambda_n)_{n \in \mathbb{Z}}$ such that 
	\[
	\lim_{|n| \to \infty} \|x_n - T^n p - \lambda_n T^n q\| = 0,
	\]
	we say that $T$ has the \emph{limit weak super-shadowing property}.
\end{definition}

\begin{lemma}\label{equi-lwss}
	An invertible operator \( T \) on a Banach space \( X \) has the limit weak super-shadowing property if and only if for every sequence \( (z_n)_{n \in \mathbb{Z}} \subset X \) satisfying \( \lim_{\vert n \vert \to \infty} \|z_n\| = 0 \), there exist a sequence \( (y_n)_{n \in \mathbb{Z}} \subset X \), a vector \( q \in X \), and a sequence of complex scalars \( (\beta_n)_{n \in \mathbb{Z}} \) such that:
	\begin{align*}
		\lim_{\vert n \vert \to \infty} \|y_n\| = 0 \quad \text{and} \quad y_{n+1} = Ty_n + z_n + \beta_n T^{n+1} q, \quad \forall n \in \mathbb{Z}.
	\end{align*}
\end{lemma}

\begin{proof}
	Suppose \( T \) satisfies the limit weak super-shadowing property. Let \( (z_n)_{n \in \mathbb{Z}} \subset X \) be a sequence such that \( \lim_{\vert n \vert \to \infty} \|z_n\| = 0 \). Fix \( x_0 \in X \) and define a sequence \( (x_n)_{n \in \mathbb{Z}} \) recursively by \( x_{n+1} = Tx_n + z_n \). Notice that \( \|x_{n+1} - Tx_n\| = \|z_n\| \to 0 \) as \( \vert n \vert \to \infty \). By the limit weak super-shadowing property, there exist \( p, q \in X \) and a sequence of complex scalars \( (\lambda_n)_{n \in \mathbb{Z}} \) such that:
	\begin{align*}
		\lim_{\vert n \vert \to \infty} \|x_n - T^n p - \lambda_n T^n q\| = 0.
	\end{align*}
	Define \( y_n := x_n - T^n p - \lambda_n T^n q \). By construction, \( \lim_{\vert n \vert \to \infty} \|y_n\| = 0 \). As in Lemma~\ref{equiv-wsupershado}, it is possible to find a sequence of complex scalars \( (\beta_n)_{n \in \mathbb{Z}} \) such that \( y_{n+1} = Ty_n + z_n + \beta_n T^{n+1} q \). 
	
	Conversely, assume the condition in the statement of the lemma holds. Consider any sequence \( (x_n)_{n \in \mathbb{Z}} \) such that \( \|x_{n+1} - Tx_n\| \to 0 \) as \( \vert n \vert \to \infty \). Define \( z_n := x_{n+1} - Tx_n \), which satisfies \( \lim_{\vert n \vert \to \infty} \|z_n\| = 0 \). By hypothesis, there exist \( q \in X \), a sequence \( (y_n)_{n \in \mathbb{Z}} \subset X \) and a sequence of complex scalars \( (\beta_n)_{n \in \mathbb{Z}} \) such that $\lim_{\vert n \vert \to \infty} \|y_n\| = 0$ and $y_{n+1} = Ty_n + z_n + \beta_n T^{n+1} q$.
	
	Thus, for each \( n \in \mathbb{Z} \), we have:
	\begin{align*}
		y_{n+1} - x_{n+1} = T(y_n - x_n) + \beta_n T^{n+1} q.
	\end{align*}
	Unfolding this recurrence relation, we have \( y_n - x_n = T^n (y_0 - x_0) - \lambda_n T^n q \) for each \( n \in \mathbb{Z} \) for some sequence of complex numbers \( (\lambda_{n})_{n} \). Put \( p = x_0 - y_0 \). Consequently, \( \|x_n - T^n p - \lambda_n T^n q\| = \|y_n\| \) converges to 0 as \( \vert n \vert \) tends to infinity. This completes the proof.
\end{proof}

Recall that an operator $T$ on a Banach space is said to be \emph{hyperbolic} if its spectrum \(\sigma(T)\) does not intersect the unit circle \(\mathbb{T}\). 

If \( T \) is an invertible hyperbolic operator, the space \( X \) admits a decomposition \( X = E \oplus F \), where \( E \) and \( F \) are \( T \)-invariant closed subspaces of \( X \). There exist constants \( \gamma \in (0,1) \) and \( C > 1 \) such that:
\begin{align}
	\label{descom1}	E & = \{y \in X : \|T^{n}y\| \leq C \gamma^{n} \|y\| \; \text{for all} \; n \in \mathbb{N}\},\ \\
	\label{descom2}	F & = \{z \in X : \|T^{-n}z\| \leq C \gamma^{n} \|z\| \; \text{for all} \; n \in \mathbb{N}\}.
\end{align}

Nilson Bernardes and Ali Messaoudi obtained equivalences for an invertible hyperbolic operator; see, for instance, \cite[Theorem 1]{bernardes2021shadowing}.

The following statement helps to distinguish the weak super-shadowing property from the super-shadowing property.

\begin{proposition}\label{weak-no super}
	Let \( X \) be a Banach space, and let \( T \in \mathcal{L}(X) \) be an invertible operator with $\mathrm{dim}(X)\geq 1$. If $T$ is hyperbolic then the operator \( S\in \mathcal{L}(X \oplus \mathbb{C}) \), defined by \(S(x, y) = (T x, \beta y) \) with $\beta\in \mathbb{T}$ has weak super-shadowing property but does not possess the super-shadowing property.
\end{proposition}

\begin{proof}
	Let \( T \) be an invertible hyperbolic operator on a Banach space \( X \) with \( \mathrm{dim}(X) \geq 1 \). Consider operator \( S(x, y) = (Tx, \beta y) \) for \( (x, y) \in X \times \mathbb{C} \) and \( \beta \in \mathbb{T} \).
	
	Let $(x_{n}, y_{n})_{n}$ be a $\delta$-pseudotrajectory of $S$. Then, there exists $a \in X$ such that $\Vert x_{n} - T^{n}(a) \Vert < \epsilon$ since $T$ has the shadowing property. We choose a complex number $b \neq 0$. As consequence,
	\begin{align*}
		\left\Vert(x_{n}, y_{n})-S^{n}(a,0)-\left(\frac{y_{n}}{b}\right)S^{n}(0, b)\right\Vert<\epsilon\hspace{0.3cm} \forall n.
	\end{align*}
	
	Now, we will show that \( S \) does not have the super-shadowing property. To establish this, consider the decomposition \( E \oplus F = X \) as mentioned earlier, satisfying (\ref{descom1}) and (\ref{descom2}). Without loss of generality, we may assume that \( F \neq \{0\} \).
	
	Suppose the operator \( S\) has the super-shadowing property. We will derive a contradiction. In this way, let \( \delta > 0 \) such that every \( \delta \)-pseudotrajectory for \( S \) is \( 1 \)-super-shadowable.
	
	Consider a \( \delta \)-pseudotrajectory \( (T^{n}p, \beta^{n}\delta \omega_{n}) \), where \( p \in F \setminus \{0\} \) and \( (\omega_{n})_{n} \) is a sequence of integers such that \( (\omega_{n})_{n>0} \), \( (\omega_{n})_{n<0} \) are unbounded and \( |\omega_{n+1} - \omega_{n}| = 1 \). Then, there exists \( (a,b) \in X \times \mathbb{C} \) and a sequence of complex numbers \( (\lambda_{n})_{n} \) such that for each \( n \in \mathbb{Z} \):
	\begin{align*}
		\| (T^{n}p, \beta^{n}\delta \omega_{n}) - \lambda_{n} (T^{n}a, \beta^{n}b) \| < 1.
	\end{align*}
	
	Thus, \( \|T^{n}p - \lambda_{n}T^{n}a\| < 1 \) and \( \|\delta \omega_{n} - \lambda_{n}b\| < 1 \) for each \( n \in \mathbb{Z} \). From the first inequality, it follows that \( a \neq 0 \), and from the second inequality, it follows that \( b \neq 0 \), which implies that \( (\lambda_n)_{n<0} \) and \( (\lambda_n)_{n>0} \) are unbounded.
	
	We claim that \( a \in F \). Suppose the contrary, i.e., \( a = \xi_{1} + \xi_{2} \) with \( \xi_{1} \in E \setminus \{0\} \) and \( \xi_{2} \in F \). Then, \( \| \lambda_{n} T^{n} \xi_{1} \| < \| \mathbb{P} \| \) for each \( n \in \mathbb{Z} \), where \( \mathbb{P} \) is the canonical projection of \( X \) onto \( E \).
	
	Using (\ref{descom1}), we obtain for \( n < 0 \):
	\begin{align*}
		\| \mathbb{P} \| > \| \lambda_{n} T^{n} \xi_{1} \| \geq \frac{1}{C} \gamma^{n} \| \lambda_{n} \| \| \xi_{1} \|.
	\end{align*}
	Clearly, the expression on the right-hand side of the inequality above is unbounded, leading to a contradiction.
	
	Since \( a \in F \), we have \( p - \lambda_{n} a \in F \) and \( \lambda_{n} a \neq p \) for infinitely many \( n > 0 \). Therefore, for these indices, we observe that
	\begin{align*}
		\frac{1}{C} \gamma^{-n} (\| \lambda_{n} a \| - \| p \|) < \| T^{n} (p - \lambda_{n} a) \| < 1.
	\end{align*}
	Clearly, the expression on the left-hand side is unbounded for infinitely many \( n \), leading to a contradiction. Thus, we conclude that \( S \) does not have the super-shadowing property.
\end{proof}

\begin{proposition}\label{lwss-no lss}
	Let $S$ be as defined in Proposition~\ref{weak-no super}. Then $S$ has the limit weak super-shadowing property but does not have the limit super-shadowing property.
\end{proposition}

\begin{proof}
	To show that \( S \) does not satisfy the limit super-shadowing property, consider the sequence \( z_{n} := (T^{n}p, \theta_{n}\beta^{n})_{n} \), where \( p \in F \), \( \theta_{0} = 0 \), and \( \theta_{n} = \sum_{k=1}^{|n|}k^{-1} \). Clearly, \( \lim \| S z_{n} - z_{n+1} \| = 0 \). If \( S \) had the limit super-shadowing property, then there would exist \( (a, b) \in X \times \mathbb{C} \) and a sequence of complex numbers \( (\lambda_{n})_{n} \) such that
	\begin{align*}
		\lim_{|n| \to \infty} \| (T^{n}p, \theta_{n}\beta^{n}) - \lambda_{n}(T^{n}a, \beta^{n}b) \| = 0.
	\end{align*}
	Certainly, \( b \neq 0 \), which implies that \( \lim \frac{\lambda_{n}}{\theta_{n}} = \frac{1}{b} \). Similarly, it can be shown that \( a \in F \setminus \{0\} \). Therefore, for sufficiently large \( n > 0 \), 
	\begin{align*}
		\| T^{n}(p - \lambda_{n}a) \| \geq \frac{1}{C} \gamma^{-n}(\| \lambda_{n}a \| - \| p \|).
	\end{align*}
	The contradiction arises from noting that the left-hand side converges to \( 0 \), while the right-hand side is unbounded.
\end{proof}

\begin{theorem}\label{super sha- no sha}
	Each Banach space \( X \) with \( \mathrm{dim}(X) \geq 1 \) admits a continuous linear operator on \( X \) that has the positive super-shadowing property but not the positive shadowing property.
\end{theorem}

\begin{proof}
	Let \( X \) be a Banach space with \( \mathrm{dim}(X) \geq 1 \). We distinguish two cases: \( \mathrm{dim}(X) = 1 \) and \( \mathrm{dim}(X) > 1 \). For the first case, consider the identity map on \( X \). For the second case, consider a one-dimensional subspace \( N \subset X \) and its topological complement \( M \), and define the operator \( T: X \to X \) by
	\begin{align*}
		T: M \oplus N &\to M \oplus N, \\
		x + y &\mapsto y.
	\end{align*}
	
	Denoted $\mathbb{P}$ the projection operator of $X$ on $M$. According to \cite{bernardes2024shadowing} $T$ do not has the shadowing shadowing property. Now, we will show that \( T \) has the positive super-shadowing property. For any \( \epsilon > 0 \), let \( \delta = \epsilon / (2 \Vert{\mathbb{P}\Vert})\) and consider \( (z_n)_n \) a \( \delta \)-pseudotrajectory for \( T \).
	
	We can observe that there exist sequences \( (x_n)_n \subset M \) and \( (y_n)_n \subset N \) such that \( z_n = x_n + y_n \) for each \( n \in \mathbb{N} \). Thus, \( z_{n+1} - T z_n = x_{n+1} + y_{n+1} - y_n \) for each \( n \geq 0 \). Therefore,
	\[
	\|x_{n+1}\|=\|\mathbb{P}(T z_n - z_{n+1})\| \leq \epsilon/2, \forall n\geq 0.
	\]
	In particular, \( \|x_n\| < \epsilon \) for each \( n \geq 1 \). Now, let us construct a vector that \( \epsilon \)-super-shadow \( (z_n)_n \). Specifically, let \( z = x_0 + \nu \), where \( \nu \in N \setminus \{0\} \) and \( \| \nu - y_0 \| < \epsilon \). Then \( \| z_0 - z \| = \| y_0 - \nu \| < \epsilon \). For \( n \geq 1 \), choose \( \lambda_n \) such that \( y_n = \lambda_n \nu \). Then
	\[
	\| z_n - \lambda_n T^n z \| = \| x_n + y_n - \lambda_n \nu \| = \| x_n \| < \epsilon.
	\]
	This completes the proof.
\end{proof}

\begin{proposition}\label{T-inv}
	Let \( T \) be an invertible operator on a Banach space \( X \), and suppose \( X = M \oplus N \), where \( M \) and \( N \) are closed \( T \)-invariant subspaces of \( X \). If \( T \) has the super-shadowing property (or the limit super-shadowing property), then \( T|_M \) and \( T|_N \) also have the same property. 
\end{proposition}

\begin{proof}
	Let \( \mathbb{P}: X \to M \) be the canonical projection operator. Suppose \( T \) has the super-shadowing property. For \( \epsilon > 0 \), let \( \delta > 0 \) be the corresponding constant such that every \( \delta \)-pseudotrajectory for \( T \) is \( \frac{\epsilon}{\|\mathbb{P}\|} \)-super-shadowable.
	
	Now, let \( (x_n)_{n} \subset M \) be a \( \delta \)-pseudotrajectory for \( T|_M \). Since \( M \) is \( T \)-invariant, \( (x_n)_n \) is also a \( \delta \)-pseudotrajectory for \( T \). By hypothesis, there exist \( a = p + q \in X \), where \( p \in M \) and \( q \in N \), and a sequence of complex scalars \( (\lambda_n)_{n} \) such that:
	\begin{align*}
		\|x_n - \lambda_n T^n a\| < \frac{\epsilon}{\|\mathbb{P}\|}, \quad \forall n.
	\end{align*}
	Applying the projection \( \mathbb{P} \), we have:
	\[
	\|x_n - \lambda_n T^n p\| \leq \|\mathbb{P}\| \|x_n - \lambda_n T^n a\| < \epsilon, \quad \forall n.
	\]
	Thus, \( T|_M \) has the super-shadowing property. A similar argument applies to \( T|_N \).
	
	To prove the statement for the limit super-shadowing property, assume \( T \) has this property. Let \( (x_n)_{n} \subset M \) be a sequence with \( \|Tx_n - x_{n+1}\|\rightarrow 0\). By hypothesis, there exist \( a = p + q \in X \), where \( p \in M \) and \( q \in N \), and a sequence of complex scalars \( (\lambda_n)_{n} \) such that:
	\[
	\|x_n - \lambda_n T^n a\| \xrightarrow[\vert{n\vert}\rightarrow \infty \;]{} 0.
	\]
	Since \( \mathbb{P}(x_n - \lambda_n T^n a) = x_n - \lambda_n T^n p \), we have:
	\[
	\|x_n - \lambda_n T^n p\| \leq \|\mathbb{P}\| \|x_n - \lambda_n T^n a\|  \xrightarrow[\vert{n\vert}\rightarrow \infty \;]{} 0.
	\]
	This shows that \( T|_M \) has the limit super-shadowing property. A similar argument applies to \( T|_N \), completing the proof.
\end{proof}

\begin{remark}
	The preceding proposition remains valid for a continuous operator $T$ satisfying the positive (limit) super-shadowing property.
\end{remark}

\begin{proposition}
	For any $T\in \mathcal{L}(X)$, the following assertions are equivalent:
	\begin{enumerate}
		\item $T$ has the (weak) super-shadowing property
		\item  $T^{k}$  has the (weak) super-shadowing property for some $k\in \mathbb{N}$
		\item $T^{k}$ has the (weak) super-shadowing property for each $k\in \mathbb{N}$.
	\end{enumerate}
\end{proposition}

\begin{proof}
	To prove (1) implies (3). Let $\epsilon>0$ and $\delta$ corresponds to the super-shadowing property of $T$. Choose any fixed positive integer $k$. We can see that if $(x_{n})_{n}$ is a $\delta$-pseudotrajectory of $T^{k}$ with $\delta>0$, then
	\begin{eqnarray*}
		(x_{0},Tx_{0}, \ldots, T^{k-1}x_{0}, x_{1}, Tx_{1}, \ldots, T^{k-1}x_{1}, x_{2}, Tx_{2}, \ldots, T^{k-1}x_{2}, \ldots )
	\end{eqnarray*}
	is a $\delta$-pseudotrajectory of $T$. Therefore, there exists a vector $(p,q)$ of $X\times X$ and a sequence of complex numbers $(\lambda_{n})_{n}$ such that 
	\begin{eqnarray*}
		\Vert{x_{n}- (T^{k})^{n}p-\lambda_{kn}(T^{k})^{n}q\Vert}<\epsilon\;\; \text{for all}\; n.
	\end{eqnarray*}
	It is evident that (3) implies (2). In order to conclude the proof, let u\textit{}s see that (2) implies (1). Suppose that $T^{k}$ has the super-shadowing property. Given $\epsilon>0$, there exists $\delta_{1}>0$ such that
	$\sum_{j=0}^{k-1}T^{j}(B(0,\delta_{1}))\subset B(0,\epsilon)$. Now, for  $\delta_{1}$ there exists $\delta_{2}>0$ with $\delta_{2}\in (0, \delta_{1})$
	corresponding to the positive super-shadowing property of $T^{k}$. We can choose $\delta>0$ such that $\sum_{j=0}^{k-1}T^{j}(B(0,\delta))\subset B(0,\delta_{2})$. If $(x_{n})_{n}$ is a $\delta$-pseudotrajectory for $T$, then $(x_{kn})_{n}$ is $\delta_{2}$-pseudotrajectory for $T^{k}$, then there exists $(p,q)\in X\times X$ and a sequence of complex numbers $(\beta_{n})_{n}$ such that
	\begin{eqnarray*}
		x_{kn}-(T^{k})^{n}p-\lambda_{n}(T^{k})^{n}q\in B(0,\delta_{1}), \hspace{0.3cm} \text{for all}\; n.
	\end{eqnarray*}
	Consider the sequence $(\beta_{j})_{j}$ given by $b_{j}=\lambda_{n}$ whenever $kn\leq j<(k+1)n-1$. Then for the $\delta$-pseudotrajectory $(x_{j})_{j}$ we have
	\begin{align*}
		\;\;\;  x_{j}-T^{j}p-\beta_{j}T^{j}q &= \sum_{\ell=0}^{j-kn-1}T^{\ell}(x_{j-\ell}-Tx_{j-\ell-1})\\
		{}  &{} +
		T^{j-kn}(x_{kn}-T^{kn}p-\lambda_{n}T^{kn}q)\in B(0,\epsilon).
	\end{align*}
This completes the proof.
\end{proof}

\section{Compact Operators and positive super-shadowing}\label{sec3}

Our first example of a compact operator that does not possess the positive shadowing property but does possess the \emph{positive super-shadowing} property is presented in Theorem \ref{super sha- no sha}. In this section, we characterize compact operators on a Banach space that possess the positive super-shadowing property (Theorem \ref{compact-1}) and the positive limit super-shadowing property (Theorem \ref{compact-2}).

Nilson Bernardes et al. showed that, for compact operators, hyperbolicity is equivalent to both the positive shadowing property and the positive limit shadowing property \cite[Theorem 15]{bernardes2018expansivity}. This allows us to conclude that if a compact operator $T$ does not possess the positive shadowing property, then its spectrum $\sigma(T)$ intersects $\mathbb{T}$.

It is a well-established result that for invertible operators in finite-dimensional spaces, hyperbolicity is equivalent to the operator having the shadowing property. This fact is proven by considering the Jordan decomposition of the operator. Assuming that the operator admits an eigenvalue on $\mathbb{T}$, a pseudotrajectory is constructed that cannot be shadowed \cite{ombach1994shadowing}.

\begin{theorem}\label{dim-finita}
	For $m > 1$, let $T: \mathbb{C}^m \to \mathbb{C}^m$ be an invertible operator. Then the following are equivalent:
	\begin{enumerate}
		\item $T$ has the shadowing property.
		\item $T$ has the super-shadowing property.
		\item $T$ has the limit super-shadowing property.
	\end{enumerate}
\end{theorem}

\begin{proof}
	The implications $(1) \implies (2)$ and $(1) \implies (3)$ are straightforward. We will first show that $(2) \implies (1)$, followed by proving that $(3) \implies (1)$.
	
	\textbf{Step 1: Proving $(2) \implies (1)$.} 
	Suppose $T$ satisfies the super-shadowing property but does not have the shadowing property. Then, $\sigma(T) \cap \mathbb{T} \neq \emptyset$. Let $\beta \in \sigma(T) \cap \mathbb{T}$ be an eigenvalue of $T$. There exists a $T$-invariant subspace $E \subset \mathbb{C}^m$ associated with this eigenvalue. By Theorem~\ref{T-inv}, the restriction $T|_E$ also satisfies the super-shadowing property.
	
	We claim that $\dim(E) = 1$. Suppose, for contradiction, that $\dim(E) = k > 1$. In this case, $T|_E$ is similar either to a diagonal matrix with entries $\beta$ or to a lower triangular Jordan matrix of the form:
	\begin{align}\label{beta-matrix}
		\begin{bmatrix}
			\beta & {} & {} & {} & 0 \\
			1 & \beta & {} & {} & {} \\
			{} & 1 & \beta & {} & {} \\
			{} & {} & \ddots & \ddots & {} \\
			0 & {} & {} & 1 & \beta
		\end{bmatrix}_{k \times k}.
	\end{align}
	
	\emph{Case 1: Diagonal Matrix.} Without loss of generality, assume $k = 2$. Given $\epsilon > 0$, there exists $\delta > 0$ associated with the super-shadowing property. Consider the $\delta$-pseudotrajectory:
	\begin{align}\label{2x2}
		\begin{bmatrix}
			\delta \beta^n n \\
			\delta \beta^n \psi(n)
		\end{bmatrix},
	\end{align}
	where $|\psi(n) - \psi(n+1)| = 1$ for each $n$, $\psi(n)$ is non-negative and unbounded, and $\{n \in \mathbb{Z} : \psi(n) = 1\}$ is infinite. By the super-shadowing property, there exist $(p, q) \in \mathbb{C}^2$ and a sequence of scalars $(\lambda_n)_{n \in \mathbb{Z}}$ such that
	\[
	\left\| \begin{bmatrix}
		\delta \beta^n n - \lambda_n \beta^n p \\
		\delta \beta^n \psi(n) - \lambda_n \beta^n q
	\end{bmatrix} \right\| < \epsilon, \quad \forall n \in \mathbb{Z}.
	\]
	This implies that $|\lambda_n| \to \infty$ as $|n| \to \infty$. On the other hand, since $q \neq 0$, there exists an infinite subset of integers where $|\lambda_n|$ is bounded above by $(\epsilon + \delta)/|q|$, leading to a contradiction.
	
	\emph{Case 2: Jordan Matrix.} Now consider the sequence $(y_n)_{n \in \mathbb{Z}}$, where $y_n$ is given by:
	\[
	\begin{bmatrix}
		\delta \beta^n n \\
		\delta \beta^{n-1} \frac{n(n-1)}{2} \\
		\vdots
	\end{bmatrix}_{k \times 1}.
	\]
	The remaining components of $y_n$ are chosen to ensure it forms a $\delta$- pseudotrajectory. By hypothesis, there exist $(p, q) \in \mathbb{C}^2$ and a sequence of scalars $(\lambda_n)_{n \in \mathbb{Z}}$ such that:
	\[
	|\delta n - \lambda_n p| < \epsilon \quad \text{and} \quad \left|\delta \frac{n(n-1)}{2} - \lambda_n n p - \lambda_n \beta q\right| < \epsilon.
	\]
	From the first inequality, we deduce that $p \neq 0$ and:
	\[
	\lim_{|n| \to \infty} \frac{\lambda_n}{n} = \frac{\delta}{p}.
	\]
	From the second inequality, we obtain:
	\[
	\left|\frac{\delta}{2} - \frac{\lambda_n}{n-1} p - \frac{\lambda_n \beta q}{n(n-1)}\right| < \frac{\epsilon}{n(n-1)}.
	\]
	Taking the limit as $|n| \to \infty$ leads to a contradiction.
	
	Suppose there exist two distinct eigenvalues of $T$, $\{\alpha, \beta\} \subset \mathbb{T}$. Then, each eigenvalue has an associated one-dimensional eigenspace. By restricting $T$ to the direct sum of these eigenspaces, $T$ inherits the super-shadowing property. Using a procedure similar to (\ref{2x2}), consider the following sequence:
	\begin{align}
		\begin{bmatrix}
			\delta \alpha^n n \\
			\delta \beta^n \psi(n)
		\end{bmatrix},
	\end{align}
	where $|\psi(n) - \psi(n+1)| = 1$ for all $n$, and $\psi(n)$ is non-negative and unbounded. This construction leads to a contradiction, as shown earlier. 
	
	Thus, $T$ admits at most one eigenvalue in $\mathbb{T}$, with its associated eigenspace being one-dimensional. Since $m > 1$, it follows that $\sigma(T) \setminus \mathbb{T} \neq \emptyset$. By Proposition~\ref{weak-no super}, $T$ cannot have the super-shadowing property. This establishes that (2) implies (1).

	\textbf{Step 2: Proving $(3) \implies (1)$.} 
	
	Suppose that $T$ satisfies condition (3) but does not have the shadowing property. We proceed similarly to the proof of $(2) \implies (1)$. This implies that $\sigma(T) \cap \mathbb{T} \neq \emptyset$. Let $\beta \in \sigma(T) \cap \mathbb{T}$, and let $E \subset \mathbb{C}^m$ be the corresponding eigenspace. We claim that $\dim(E) = 1$. Suppose, for contradiction, that $\dim(E) = k > 1$. Then, $T|_E$ has the super-shadowing property and is similar either to a diagonal matrix with entries $\beta$ or to a lower triangular Jordan matrix as in (\ref{beta-matrix}).
	
	\emph{Case 1: Diagonal Matrix}. Without loss of generality, assume $k = 2$. Consider the sequence $(x_n)_{n}$ given by:
	\begin{align*}
		\begin{bmatrix}
			\delta \beta^n \theta(n) \\
			\delta \beta^n \psi(n)
		\end{bmatrix},
	\end{align*}
	where $\theta(n) = \sum_{k=1}^n \frac{1}{k}$ and $\psi(n) = \sum_{k=1}^n \frac{1}{k^{1/2}}$ for $n > 1$. By hypothesis, there exists $(p, q) \in \mathbb{C}^2$ such that:
	\begin{align*}
		\lim_{n \to \infty} |\delta \beta^n \theta(n) - \beta^n \lambda_n p| = 0 \quad \text{and} \quad \lim_{n \to \infty} |\delta \beta^n \psi(n) - \beta^n \lambda_n q| = 0.
	\end{align*}
	Thus, $p, q \neq 0$, and we can distinguish two asymptotic behaviors:
	\begin{align*}
		\lim_{n \to \infty} \frac{\lambda_n}{\theta(n)} = \frac{\delta}{p} \quad \text{and} \quad \lim_{n \to \infty} \frac{\psi(n)}{\lambda_n} = \frac{q}{\delta}.
	\end{align*}
	This leads to a contradiction because $\psi(n)/\theta(n)$ is unbounded as $n \to \infty$.
	
	\emph{Case 2: Jordan Matrix.}. Consider the sequence $(x_n)_{n}$ given by:
	\begin{align*}
		x_n := (\beta^n \gamma_{1,n}, \beta^{n-1} \gamma_{2,n}, \dots, \beta^{n-k+1} \gamma_{k,n}),
	\end{align*}
	for $n > 0$, where $\gamma_{1,n} = \sum_{k=1}^n \frac{1}{k}$, $\gamma_{2,n+1} = \sum_{k=1}^n \gamma_{1,k}$, and $\gamma_{i,n}$ is defined recursively for $2 \leq i \leq k$. By hypothesis and examining the first two components, there exist $(p, q) \in \mathbb{C}^2$ and a sequence $(\lambda_n)_n \subset \mathbb{C}$ such that:
	\begin{align*}
		\lim_{n \to \infty} &\|\gamma_{1,n+1} - \lambda_{n+1} p\| = 0, \\
		\lim_{n \to \infty} &\|\gamma_{2,n+1} - \lambda_{n+1}(n+1)p - \beta \lambda_{n+1} q\| = 0.
	\end{align*}
	Note that $\gamma_{2,n+1} = (n+1)\gamma_{1,n+1} - (1+n)$. Substituting this into the second equation, we obtain:
	\begin{align*}
		\|\gamma_{2,n+1} - \lambda_{n+1}(n+1)p - \beta \lambda_{n+1} q\| = (n+1)\|\gamma_{1,n+1} - \lambda_{n+1}p - 1 - \frac{\beta \lambda_{n+1}}{n+1}q\|.
	\end{align*}
	This implies:
	\begin{align*}
		\lim_{n \to \infty} \|\gamma_{1,n+1} - \lambda_{n+1}p - 1 - \frac{\beta \lambda_{n+1}}{n+1}q\| = 0,
	\end{align*}
	which is a contradiction. 
	
	One can observe that $T$ admits at most one eigenvalue in $\mathbb{T}$, with its associated eigenspace being one-dimensional. Since $m > 1$, it follows that $\sigma(T) \setminus \mathbb{T} \neq \emptyset$. By Proposition~\ref{lwss-no lss}, $T$ cannot have the limit super-shadowing property. This establishes that (3) implies (1).
\end{proof}

\begin{theorem}\label{compact-1}
	Let $T$ be a compact operator on a complex Banach space $X$ with \(\text{dim}(X)\geq 2\). The following statements are equivalent:
	\begin{enumerate}
		\item $T$ satisfies the \textit{positive super-shadowing} property but does not satisfy the \textit{positive shadowing} property.
		\item The dynamical system $(X, T)$ admits a decomposition given by the existence of a closed subspace $Y$ and a one-dimensional subspace $N$, both $T$-invariant, such that $X = Y \oplus N$. Moreover, $T|_Y: Y \to Y$ is nilpotent, and $T|_N: N \to N$ is given by $T|_N(y) = \beta y$ for some $\beta \in \mathbb{T}$.
	\end{enumerate}
\end{theorem}

\begin{proof}
	
	We will show that (1) implies (2). Suppose $T: X \to X$ satisfies the assumptions of (1). Note that \(\sigma(T)\cap \mathbb{T}\neq \emptyset\) since $T$ does not have the positive shadowing property. We claim that \(\sigma(T)\setminus \mathbb{T}\neq \emptyset\). Suppose otherwise, i.e., \(\sigma(T)\subset \mathbb{T}\), which implies that $T:X\to X$ is invertible. Since $T$ is a compact operator, it follows that \(\text{dim}(X)<\infty\). By Theorem \ref{dim-finita}, we obtain \(\text{dim}(X)=1\), leading to a contradiction.  
	
	In this case, we can partition its spectrum $\sigma(T)$ into the union of two disjoint closed non-empty sets:
	\[
	\sigma(T) = \sigma_1(T) \cup \sigma_2(T),
	\]
	where
	\[
	\sigma_1(T) = \{\lambda \in \sigma(T) : |\lambda| \neq 1\} \quad \text{and} \quad \sigma_2(T) = \{\lambda \in \sigma(T) : |\lambda| = 1\}.
	\]
	
	By the Riesz Decomposition Theorem, there exist two closed $T$-invariant subspaces $Y$ and $N$ of $X$ such that $X = Y \oplus N$, $\sigma(T|_Y) = \sigma_1(T)$, and $\sigma(T|_N) = \sigma_2(T)$. Let $S := T|_Y$.
	
	Since $T$ is compact, it follows that $\dim(N) < \infty$. By Theorem \ref{dim-finita}, $N$ is one-dimensional. Thus, $T|_N: N \to N$ is given by $T|_N(y) = \beta y$ for some $\beta \in \mathbb{T}$. Denote by $\mathbb{P}$ the canonical projection from $X$ onto $N$, and by $\mathbb{P}_Y$ the projection from $X$ onto $Y$. Furthermore, let $\nu \in N$ be a unit vector ($\|\nu\| = 1$).
	
	We now show that $S$ is a nilpotent operator. Note that $\sigma(S) = \{0\}$, which follows from the Riesz Decomposition Theorem, Theorem \ref{weak-no super}, and Proposition \ref{T-inv}.
	
	For $\epsilon = 1$, let $\delta > 0$ be associated with the positive super-shadowing property of $T$. Fix $p \in Y$ and consider the $\delta$-pseudotrajectory for $T$ given by
	\[
	\left((S^n p, (2\|\mathbb{P}\| + n\delta)\beta^n \nu)\right)_n.
	\]
	Thus, there exists $(q, b) \in Y \oplus N$ and a sequence of complex numbers $(\lambda_n)_n$ such that $\lambda_0 = 1$ and
	\[
	\|(S^n p, (2\|\mathbb{P}\| + n\delta)\beta^n \nu) - \lambda_n (S^n q, \beta^n b)\| < 1, \quad \forall n \geq 0.
	\]
	First, consider the case $n = 0$, which yields:
	\[
	\|\mathbb{P}\| < \|b\| < 3\|\mathbb{P}\|.
	\]
	For $n > 0$, we obtain the following inequalities:
	\[
	\frac{\|\mathbb{P}\| + n\delta}{3\|\mathbb{P}\|} < |\lambda_n| < \frac{3\|\mathbb{P}\| + n\delta}{\|\mathbb{P}\|}.
	\]
	Let $n_1 \in \mathbb{N}$ be such that $n\delta > 2\|\mathbb{P}\|$ for all $n > n_1$.
	
	The above estimates were established by projecting onto the subspace $N$. Since $S$ is compact and $\sigma(S) = \{0\}$, there exists $n_2 \in \mathbb{N}$ such that $\|S^n x\| < 1$ for all $\|x\| < 1$ and for each $n > n_2$.
	
	Note that $\|p - q\| < \|\mathbb{P}_Y\|$. For $n > \max\{n_1, n_2\}$, we have the following chain of inequalities:
	\begin{align*}
		\|S^n p - \lambda_n S^n q\| &< \|\mathbb{P}_Y\|, \\
		(|\lambda_n| - 1)\|S^n q\| &< \|\mathbb{P}_Y\| + \|S^n(p - q)\|, \\
		\left(\frac{n\delta - 2\|\mathbb{P}\|}{3\|\mathbb{P}\|}\right)\|S^n q\| &< 2\|\mathbb{P}_Y\|, \\
		\|S^n q\| &< \frac{6\|\mathbb{P}\|\|\mathbb{P}_Y\|}{n\delta - 2\|\mathbb{P}\|}.
	\end{align*}
	Thus, for $m > \max\{n_1, n_2\}$, we have:
	\begin{align*}
		\|S^m p\| &< \|\mathbb{P}_Y\| + |\lambda_m|\|S^m q\|, \\
		&< \|\mathbb{P}_Y\| + 6\frac{(3\|\mathbb{P}\| + m\delta)}{m\delta - 2\|\mathbb{P}\|}\|\mathbb{P}_Y\|.
	\end{align*}
	This implies that there exists $M > 0$ such that $\|S^m p\| < M$ for all $p \in Y$. Hence, $S^m \equiv 0$, proving that $S$ is nilpotent.
	
	Now, we will prove that (2) implies (1). Let \(\mathbb{P}\) and \(\mathbb{P}|_{Y}\) be the canonical projections onto \(N\) and \(Y\), respectively. Since \(S := T|_{Y}\) is nilpotent, there exists \(m \in \mathbb{N}\) such that \(S^{m} \equiv 0\). Thus, for any \(\epsilon > 0\), we define
	\begin{align*}
		\delta := \frac{\epsilon}{\Vert \mathbb{P}|_{Y} \Vert \sum_{\ell=0}^{m-1} \Vert S^{\ell} \Vert + m \Vert \mathbb{P} \Vert+1}.
	\end{align*}
	Let \(((x_{n},y_{n}))_{n \geq 0}\) be a \(\delta\)-pseudotrajectory for \(T\). Consider \(q \in N \setminus \{0\}\) such that \(\Vert y_{0} - q \Vert < \delta\), and define a sequence of complex numbers \((\lambda_{n})_{n \geq 0}\) by setting \(\lambda_{n} = 1\) for \(0 \leq n \leq m-1\) and ensuring that \(y_{n} = \lambda_{n} \beta^{n} q\) for all \(n \geq m\).
	
	First, observe that for \(1 \leq n \leq m-1\), we have
	\begin{align*}
		\Vert (x_{n}, y_{n}) - T^{n} (x_{0}, q) \Vert & \leq \Vert x_{n} - S^{n} x_{0} \Vert + \Vert y_{n} - \beta^{n} q \Vert \\
		& \leq \sum_{i=0}^{n-1} \left( \Vert S^{i} (x_{n-i} - S x_{n-i-1}) \Vert + \Vert y_{n-i} - \beta y_{n-i-1} \Vert \right) \\
		& \hspace{3.5cm} + \Vert y_{0} - q \Vert \\
		& < \delta \left( \Vert \mathbb{P}|_{Y} \Vert \sum_{i=0}^{m-1} \Vert S^{i} \Vert + \Vert \mathbb{P} \Vert m + 1 \right) = \epsilon.
	\end{align*}
	Finally, for \(n \geq m\), we obtain
	\begin{align*}
		\Vert (x_{n}, y_{n}) - \lambda_{n} T^{n} (x_{0}, q) \Vert &= \Vert x_{n} \Vert \\
		& \leq \Vert S^{n} x_{0} \Vert + \sum_{i=0}^{n-1} \Vert S^{i} (x_{n-i} - S x_{n-i-1}) \Vert \\
		& \leq 0 + \Vert \mathbb{P}|_{Y} \Vert \delta \sum_{i=0}^{n-1} \Vert S^{i} \Vert \\
		& < \epsilon.
	\end{align*}
	This completes the proof.
\end{proof}

\begin{theorem}\label{compact-2}
	Let $T$ be a compact operator on a complex Banach space $X$ with \(\text{dim}(X)\geq 2\). The following statements are equivalent:
	\begin{enumerate}
		\item $T$ satisfies the \textit{positive limit super-shadowing} property but does not satisfy the \textit{positive shadowing} property.
		\item The dynamical system $(X, T)$ admits a decomposition given by the existence of a closed subspace $Y$ and a one-dimensional subspace $N$, both $T$-invariant, such that $X = Y \oplus N$. Moreover, \(\sigma(T|_{Y})=\{0\}\), and $T|_N: N \to N$ is given by $T|_N(y) = \beta y$ for some $\beta \in \mathbb{T}$.
	\end{enumerate}
\end{theorem}

\begin{proof}
	We will show that (1) implies (2). By a similar argument as in the proof of Theorem \ref{compact-1}, we can partition $\sigma(T) = \sigma_1(T) \cup \sigma_2(T)$ into two disjoint non-empty closed sets, where $\sigma_1(T) = \sigma(T) \setminus \mathbb{T}$ and $\sigma_2(T) = \sigma(T) \cap \mathbb{T}$. By the Riesz Decomposition Theorem, there exist two closed $T$-invariant subspaces $Y$ and $N$ of $X$ such that $X = Y \oplus N$, $\sigma(T|_Y) = \sigma_1(T)$, and $\sigma(T|_N) = \sigma_2(T)$. 
	
	It is worth noting that \(\text{dim}(N)<\infty\). By Proposition \ref{T-inv}, $T|_{N}: N\to N$ is invertible. According to Theorem \ref{dim-finita}, $N$ is a one-dimensional subspace. Furthermore, \(\sigma_{2}(T)=\{\beta\}\) for some \(\beta\in \mathbb{T}\).
	
	We claim that \(\sigma(T|_{Y})=\{0\}\). Suppose otherwise. In this case, we apply the Riesz Decomposition Theorem to obtain a finite-dimensional $T$-invariant subspace $M$ such that $T|_{M}:M\to M$ is invertible and hyperbolic. By considering the operator $T$ restricted to $M\oplus N$ under Proposition \ref{T-inv} and Proposition \ref{lwss-no lss}, we arrive at a contradiction. 
	
	Now we establish the converse. Suppose that $(X,T)$ satisfies the conditions of item (2). Clearly, the operator $T$ does not have the positive shadowing property since \(\sigma(T)\cap \mathbb{T}\neq \emptyset\). Let $\mathbb{P}$ and $\mathbb{P}|_{Y}$ be the canonical projections onto $N$ and $Y$, respectively. Define $S:=T|_{Y}$ with \(\sigma(S)=\{0\}\). Furthermore, let $\nu \in N$ be a unit vector.
	
	Let \(((x_{n}, y_{n}))_{n\in \mathbb{N}}\) be a sequence in $X$ satisfying
	\begin{align*}
		\lim_{n\to \infty}\Vert{ (x_{n+1},y_{n+1})-T(x_{n},y_{n})\Vert}=0.
	\end{align*}
	We immediately obtain
	\begin{align}
		\lim_{n\to\infty}\Vert{y_{n+1}-\beta y_{n}\Vert}&=0, \label{lim1}\\
		\lim_{n\to\infty} \Vert{x_{n+1}-Sx_{n}\Vert}&=0. \label{lim2}
	\end{align}
	
	Let $(\lambda_{n})_{n\geq 0}$ be a sequence of complex numbers such that $y_{n}=\lambda_{n} \nu$ for each $n\geq 0$. From \eqref{lim1}, we can verify that $|\lambda_{n}|\leq Cn$ for some $C>0$ and all $n>n_{0}$ for some $n_{0}\in \mathbb{N}$.
	
	Thus, for $n>n_{0}$,
	\begin{align*}
		\Vert{(x_{n},y_{n})-\lambda_{n}T^{n}(x_{0},\nu)\Vert}&= \Vert{x_{n}-\lambda_{n}S^{n}x_{0}\Vert}\\ 
		&\leq  \Vert{x_{n}-S^{n}x_{0}\Vert}+ |\lambda_{n}-1|\Vert{S^{n}x_{0}\Vert}\\
		& \leq \sum_{i=0}^{n-1} \Vert{S^{i}\Vert} \Vert{x_{n-i}-Sx_{n-i-1}\Vert}+2Cn\Vert{S^{n}x_{0}\Vert}.
	\end{align*}
	Since \(\sigma(S)=\{0\}\), it follows that $2Cn\Vert{S^{n}x_{0}\Vert}$ converges to $0$ as $n$ tends to infinity. Notice that $\sum_{i\in \mathbb{N}}\Vert{S^{i}\Vert}<\infty$ and, from \eqref{lim2}, we obtain
	\begin{align*}
		\sum_{i=0}^{n-1} \Vert{S^{i}\Vert} \Vert{x_{n-i}-Sx_{n-i-1}\Vert}\xrightarrow[n\to\infty]{}0.
	\end{align*}
	This proves that $T$ has the positive limit super-shadowing property.
\end{proof}



\section{Supercyclicity}
One can observe that an isometric operator on a Banach space cannot be hypercyclic. In the context of supercyclic operators, Ansari and Bourdon \cite{ansari1997some} established that no surjective isometric operator on a Banach space $X$ with $\dim(X) > 1$ can be supercyclic.

Recently, Mayara Braz et al. \cite{antunes2022chain} showed that no surjective isometric operator on a separable Banach space possesses the positive shadowing property. This motivates us to establish the following result.

\begin{theorem}\label{iso-no-super}
No surjective isometric operator on a separable Banach space $X$ with $\dim(X) > 1$ possesses the positive super-shadowing property.
\end{theorem}

To prove the above result, we need to establish some preliminary statements and recall the notion of chain recurrence.

In the context of linear dynamics \cite{antunes2022chain, bernardes2024shadowing}, it is known that if $T: X \to X$ is chain recurrent, then for any $x, y \in X$ and every $\delta > 0$, there exists a finite sequence $\{x_0, \ldots, x_n\}$ with $n > 0$, $x_0 = x$, and $x_n = y$, such that 
\begin{align}\label{cadena}
d(Tx_i, x_{i+1}) < \delta \quad \text{for all } 0 \leq i < n.
\end{align}

\begin{proposition}[Proposition 2.4 \cite{antunes2022chain}]\label{isometry}
Let $X$ be a normed vector space, and let $T: X \to X$ be a bounded linear operator that is a surjective isometry. Then $T$ is chain recurrent.
\end{proposition}

\begin{proposition}\label{supercyclic}
Let $X$ be a separable Banach space, and let $T \in \mathcal{L}(X)$ be an operator possessing both the chain recurrence and positive super-shadowing properties. Then, for any two non-empty open sets $U, V \subset X$, there exists a sequence of complex numbers $(\lambda_n)_n$ and a positive integer $n_0$ such that
\[
\lambda_n T^n(U) \cap V \neq \emptyset, \quad \forall n \geq n_0.
\]
\end{proposition}

\begin{proof}
Let $U, V$ be two non-empty open subsets of $X$. Choose $x \in U$ and $y \in V$, and let $r > 0$ be such that $B(x, r) \subset U$ and $B(y, r) \subset V$. Fix $\epsilon = r/2$, and let $\delta > 0$ be the corresponding value associated with the positive super-shadowing property.

Since $T$ is chain recurrent, there exist two finite sequences satisfying (\ref{cadena}):
\begin{itemize}
    \item $\{x_0, x_1, \dots, x_\ell\}$ with $x_0 = x$ and $x_\ell = 0$,
    \item $\{y_0, y_1, \dots, y_m\}$ with $y_0 = 0$ and $y_m = y$.
\end{itemize}

For each $k \geq 1$, construct a $\delta$-pseudotrajectory of the form:
\[
(z_i^k)_i := \{x_0, x_1, \dots, x_{\ell-1}, \underbrace{0, \dots, 0}_{k \text{ times}}, y_1, \dots, y_m, Ty_m, T^2 y_m, \dots\}.
\]

By hypothesis, for each $k$, there exist $q_k \in X$ and a sequence of complex numbers $(\beta_{n,k})_n$ such that:
\[
\|z_i^k - \beta_{i,k} T^i q_k \| < \epsilon, \quad \forall i \geq 0.
\]

In particular, for $i = 0$ and $i = \ell + k + m$, we have:
\[
\|x - \beta_{0,k} q_k \| < \epsilon \quad \text{and} \quad \|y - \beta_{\ell + k + m, k} T^{\ell + k + m} q_k \| < \epsilon.
\]

Assume $\beta_{0,k} \neq 0$. Consequently, $q_k \in \frac{1}{\beta_{0,k}} U$, and $\beta_{\ell + k + m, k} T^{\ell + k + m} q_k \in V$. Therefore, for $n \geq \ell + m$, define $\lambda_n := \frac{\beta_{n, n-\ell-m}}{\beta_{0, n-\ell-m}}$. It follows that:
\[
\lambda_n T^n(U) \cap V \neq \emptyset, \quad \forall n \geq \ell + m.
\]

Thus, the result is established with $n_0 = \ell + m$.
\end{proof}

\begin{proof}[Proof of Theorem \ref{iso-no-super}]
Suppose that there exists an operator $T$ on a separable Banach space that possesses the super-shadowing property and is also a surjective isometric operator. By Proposition \ref{isometry}, $T$ is chain recurrent. According to Proposition \ref{supercyclic}, $T$ is a supercyclic operator. This leads to a contradiction.
\end{proof}

An important result regarding hypercyclic operators on separable, infinite-dimensional $F$-spaces is the well-known \emph{Birkhoff's Transitivity Theorem} \cite{bayart2009dynamics, grosse2011linear}, which relies on the Baire Category Theorem. Similarly, an analogous result can be established for supercyclic operators \cite[Theorem 1.12]{bayart2009dynamics}.

Recently, Bonilla and Grosse-Erdmann proved a Birkhoff-type theorem for upper Furstenberg families \cite[Theorem 3.1]{bonilla2018upper}. This motivates us to explore this aspect in the context of supercyclic operators.

A non-empty family \(\mathcal{F}\) of subsets of \(\mathbb{N}_0\) is called a \textit{Furstenberg family} if it is hereditary upward, meaning that 
	\begin{align*}
		A \in \mathcal{F}, \, A \subset B \implies B \in \mathcal{F}.
	\end{align*}
	Moreover, a Furstenberg family is said to be \textit{proper} if it does not contain the empty set.

\begin{definition}[\cite{alves2024frequently}]
	Let $X$ be a separable $F$-space, and let $T: X \to X$ be a continuous linear operator. Consider a proper Furstenberg family $\mathcal{F}$. We say that $x \in X$ is an $\mathcal{F}$-supercyclic vector for $T$ if, for every non-empty open subset $U \subset X$, we have
	\begin{align*}
		\{n \geq 0 : \mathbb{C} \cdot T^n x \cap U \neq \emptyset\} \in \mathcal{F}.
	\end{align*}
	We denote by $\mathcal{F}\text{SC}(T)$ the set of all $\mathcal{F}$-supercyclic vectors for $T$.
\end{definition}

Now, let us recall the notion of an upper Furstenberg family. We refer the reader to \cite{bonilla2018upper} for further details.

\begin{definition}
	A Furstenberg family $\mathcal{F}$ is called \emph{upper} if it is proper and it can be written as $\mathcal{F}=\bigcup_{t\in D}\mathcal{F}_{t}$ with $\mathcal{F}_{t}:=\bigcap_{\mu\in M}\mathcal{F}_{t,\mu}$ for some families $\mathcal{F}_{t, \mu}$ where $t\in D$ (an arbitrary set) and $\mu\in M$ (a countable set) such that:
	\begin{itemize}
		\item Each family $\mathcal{F}_{t, \mu}$ is finitely hereditary upward, that is, for any $A\in \mathcal{F}_{t, \mu}$, there is a finite set $F\subset \mathbb{N}$ such that
		\begin{align}\label{f.h.u}
			A\cap F \subset B \implies B\in \mathcal{F}_{t,\mu}.
		\end{align}
		\item $\mathcal{F}$ is uniformly left-invariant, that is, for any $A\in \mathcal{F}$, there is some $t\in D$ such that for all $n\geq 0$,
		\begin{align}\label{u.l.i}
			A-n:=\{k-n: k\in A, k\geq n\}\in \mathcal{F}_{t}.
		\end{align}
	\end{itemize}
\end{definition}

\begin{example}
	We can distinguish important Furstenberg families that are upper Furstenberg families, namely:
	\begin{enumerate}
		\item The family $\mathcal{F}_{\text{ud}}$ of sets $A \subset \mathbb{N}_0$ with positive upper density, defined as:
		\begin{align*}
			\overline{\text{dens}}(A) := \limsup_{N \to \infty} \frac{|A \cap [0, N]|}{N+1} > 0.
		\end{align*}
		We say that $T: X \to X$ is \emph{upper frequently supercyclic} if there exists a vector $x \in X$ such that for every non-empty open subset $U \subset X$,
		\begin{align*}
			\{n \geq 0 : \mathbb{C} \cdot T^n x \cap U \neq \emptyset\} \in \mathcal{F}_{\text{ud}}.
		\end{align*}
		The set of all such vectors is denoted by $\text{UFSC}(T)$.
		
		\item The family $\mathcal{F}_{\text{uBd}}$ of sets $A \subset \mathbb{N}_0$ with positive upper Banach density, defined as:
		\begin{align*}
			\overline{\text{Bd}}(A) := \inf_{N \geq 0} \sup_{m \geq 0} \frac{|A \cap [m, m+N]|}{N+1}.
		\end{align*}
		We say that $T: X \to X$ is \emph{reiteratively supercyclic} if there exists a vector $x \in X$ such that for every non-empty open subset $U \subset X$,
		\begin{align*}
			\{n \geq 0 : \mathbb{C} \cdot T^n x \cap U \neq \emptyset\} \in \mathcal{F}_{\text{uBd}}.
		\end{align*}
		The set of all such vectors is denoted by $\text{RSC}(T)$.
	\end{enumerate}
\end{example}

The following result adapts \cite[Theorem 1.12]{bayart2009dynamics} to our setting, incorporating the necessary modifications.

\begin{theorem}\label{uff-super}
	Let \( T \) be a continuous linear operator on a separable \(F\)-space \( X \). Let 
    \( \displaystyle{\mathcal{F} = \bigcup_{t \in D} \bigcap_{\mu \in M} \mathcal{F}_{t, \mu}} \) be an upper Furstenberg family. The following statements are equivalent:
	\begin{enumerate}
		\item For any non-empty open subset \( V \subset X \), there exists \( t \in D \) such that for any non-empty open subset \( U \subset X \), there exists \( x \in U \) satisfying
		\[
		\{n \geq 0 : \mathbb{C} \cdot T^n x \cap V\neq \emptyset \} \in \mathcal{F}_t.
		\]
		\item For any non-empty open subset \( V \subset X \), there exists \( t \in D \) such that for any non-empty open subset \( U \subset X \) and any \( \mu \in M \), there exists \( x \in U \) satisfying
		\[
		\{n \geq 0 : \mathbb{C} \cdot T^n x \cap V\neq \emptyset \} \in \mathcal{F}_{t, \mu}.
		\]
		\item The set of \( \mathcal{F} \)-supercyclic vectors for \( T \) is residual in \( X \).
		\item \( T \) admits an \( \mathcal{F} \)-supercyclic vector.
	\end{enumerate}
\end{theorem}

\begin{proof}
	The implications \((1) \Rightarrow (2)\) and \((3) \Rightarrow (4)\) are immediate. In order to show \((2) \Rightarrow (3)\), consider a countable basis of non-empty open sets \((V_k)_{k \geq 1}\) in \(X\). For each \(k \geq 1\), let \(t_k \in D\) denote the parameter associated with the open set \(V_k\) by \((2)\). Thus, for $k\in \mathbb{N}$ and $\mu\in M$, consider the set $\mathcal{O}_{k, \mu}$ given by
	\begin{align*}
		\mathcal{O}_{k, \mu}:=\{x\in X: \{n\geq 0: \mathbb{C}\cdot T^{n}x\cap V_{k}\neq \emptyset\}\in \mathcal{F}_{t_{k}, \mu}\}.
	\end{align*}
	One can notice that $\mathcal{O}_{k, \mu}$ is dense by hypothesis. Now, we will show that $\mathcal{O}_{k,\mu}$ is an open set. For any fixed vector $x\in \mathcal{O}_{k, \mu}$, consider the following set $B:=\{n\geq 0: \mathbb{C}\cdot T^{n}x\cap V_{k}\neq \emptyset\}\in \mathcal{F}_{t_{k}, \mu}$. Recall that $\mathcal{F}_{t_{k},\mu }$ is finitely hereditary upward. Thus, there is a finite set $F\subset \mathbb{N}$ associated with $B$ that satisfies \eqref{f.h.u}. One can notice that there exist complex numbers $\{\lambda_{m}\}_{m\in B\cap F}$ such that $\lambda_{m}T^{m}x\in V_{k}$ for each $m\in B\cap F$. By a continuity argument, there exists an open set $W$ containing $x$ such that $\lambda_{m}T^{m}(W)\subset V_{k}$ for each $m\in B\cap F$. As a consequence,
	\begin{align*}
		B\cap F\subset \{n\geq 0: \mathbb{C}\cdot T^{n}y\cap V_{k}\neq \emptyset\},\; \forall y\in W.
	\end{align*}
	According to \eqref{f.h.u}, we have that $\{n\geq 0: \mathbb{C}\cdot T^{n}y\cap V_{k}\neq \emptyset\}\in \mathcal{F}_{t_{k}, \mu}$ for each $y\in W$. This shows that $\mathcal{O}_{k,\mu}$ is open, since $x\in W\subset \mathcal{O}_{k, \mu}$.
	
	By the Baire Category Theorem, $E:=\bigcap_{k, \mu} \mathcal{O}_{k, \mu}$ is a residual set. Note that each vector of $E$ is an $\mathcal{F}$-supercyclic vector for $T$.
	
	Now we will show that \((4) \Rightarrow (1)\). Let \(x \in X\) be an \(\mathcal{F}\)-supercyclic vector for \(T\), and let \(V\) be a non-empty open set in \(X\). Then 
	\begin{align*}
		\{n \geq 0 : \mathbb{C} \cdot T^n x \cap V \neq \emptyset\} \in \mathcal{F}.
	\end{align*}
	According to \eqref{u.l.i}, there exists \(t \in D\) such that 
	\begin{align*}
		\{n \geq 0 : \mathbb{C} \cdot T^n x \cap V \neq \emptyset\} - m \in \mathcal{F}_t, \quad \forall m \geq 0.
	\end{align*}
	On the other hand, since \(\mathcal{F}\) is proper, there exists \(m \geq 0\) and $\lambda\in \mathbb{C}$ such that \(\lambda T^m x \in U\). We can certainly assume that \(\lambda\neq 0\). Define \(y := \lambda T^m x \in U\), so that
	\begin{align*}
		\{n \geq 0 : \mathbb{C} \cdot T^n y \cap V\neq \emptyset\} = \{\ell \geq 0 : \mathbb{C} \cdot T^\ell x \cap V\neq \emptyset\} - m \in \mathcal{F}_t.
	\end{align*}
	This allows us to conclude the proof.
\end{proof}

As an immediate consequence of the preceding statement, we obtain the following result for the family $\mathcal{F}_{\text{ud}}$. A similar result can be established for $\mathcal{F}_{\text{uBd}}$.

\begin{corollary}
	Let \(T\) be a continuous linear operator on a separable \(F\)-space \(X\). Then the following assertions are equivalent:
	\begin{enumerate}
		\item For any non-empty open subset \(V\) of \(X\), there exists some \(t > 0\) such that for any non-empty open subset \(U\) of \(X\), there is some \(x \in U\) such that
		\[
		\overline{\emph{dens}}(n \in \mathbb{N} : \mathbb{C} \cdot T^n x \cap V \neq \emptyset) > t.
		\]
		\item For any non-empty open subset \(V\) of \(X\), there exists some \(t > 0\) such that for any non-empty open subset \(U\) of \(X\) and any \( n\geq 0\), there exist some \(x \in U\) and \(N \geq n\) such that
		\[
		\frac{1}{N+1} \mathrm{card} \{n \leq N : \mathbb{C} \cdot T^n x \cap V \neq \emptyset\} > t.
		\]
		\item \(T\) is upper frequently supercyclic.
	\end{enumerate}
	In particular, \(\mathrm{UFSC}(T)\) is residual.
\end{corollary}

\begin{corollary}
	Let \(T\) be a continuous linear operator on a separable \(F\)-space \(X\). Then the following assertions are equivalent:
	\begin{enumerate}
		\item For any non-empty open subset \(V\) of \(X\), there exists some \(t > 0\) such that for any non-empty open subset \(U\) of \(X\), there is some \(x \in U\) such that
		\[
		\overline{\emph{Bd}}(n \in \mathbb{N} : \mathbb{C} \cdot T^n x \cap V \neq \emptyset) > t.
		\]
		\item For any non-empty open subset \(V\) of \(X\), there exists some \(t > 0\) such that for any non-empty open subset \(U\) of \(X\) and any \(N \geq 0\), there exist some \(x \in U\) and \(m\geq 0\) such that
		\[
		\frac{1}{N+1} \mathrm{card} \{n\in [m,m+N] : \mathbb{C} \cdot T^n x \cap V \neq \emptyset\} > t.
		\]
		\item \(T\) is reiteratively supercyclic.
	\end{enumerate}
	In particular, \(\mathrm{RSC}(T)\) is residual.
\end{corollary}

Bès et al. \cite{bes2016recurrence} established that the set of reiteratively hypercyclic vectors coincides with the set of hypercyclic vectors. This leads us to establish a similar result for supercyclic operators, as follows:

\begin{theorem}
If $T$ is reiteratively supercyclic, then its set of reiteratively supercyclic vectors coincides with its set of supercyclic vectors, that is:
\[
\mathrm{RSC}(T) = \mathrm{SC}(T).
\]
\end{theorem}

\begin{proof}
	It is sufficient to show that every supercyclic vector is, in fact, reiteratively supercyclic. To prove this, consider \(y \in X\), a reiteratively supercyclic vector for \(T\). For any non-empty open set \(U \subset X\), the set
	\begin{align*}
		S(y,U) := \{n \in \mathbb{N} : \mathbb{C} \cdot T^n y \cap U \neq \emptyset\}
	\end{align*}
	has positive upper Banach density. 
	
	For each \(m \in S(y,U)\), there exists a non-zero complex number \(\lambda_m\) such that \(\lambda_m T^m y \in U\). In other words, \(y \in \lambda_m^{-1} T^{-m}(U)\). Now, consider the non-empty open set \(U_n\) containing \(y\) defined as
	\begin{align*}
		U_n := \bigcap_{j \in S(y,U) \cap [0,n]} \lambda_j^{-1} T^{-j}(U).
	\end{align*}
	
	For any fixed \(x \in \mathrm{SC}(T)\), for each \(n\), there exist \(k_n \in \mathbb{N}\) and \(\lambda \in \mathbb{C}\) such that \(\lambda T^{k_n} x \in U_n\). Consequently, \(\lambda \lambda_j T^{k_n + j} x \in U\) for all \(j \in S(y,U) \cap [0,n]\). As an immediate consequence, 
	\begin{align*}
		k_n + \big(S(y,U) \cap [0,n]\big) \subset \{m \geq 0 : \mathbb{C} \cdot T^m x \cap U \neq \emptyset\}=:S(x,U)
	\end{align*}
	Therefore, \(\overline{\mathrm{BD}}(S(x,U)) > 0\). This completes the proof.
\end{proof}

Recently, Ernst, Esser, and Menet showed that if $T$ is upper frequently hypercyclic, then for any $n \geq 1$, the $n$-fold product $T \times \cdots \times T$ is U-frequently hypercyclic. Similarly, the same holds when $T$ is reiteratively hypercyclic. See Corollary 2.7 and Corollary 2.8 in \cite{ernst2021u}.

In the context of supercyclic operators, we encounter a different situation. We can easily construct an example to illustrate this. Let $T$ be a U-frequently hypercyclic operator on a separable Banach space $X$. Now, define the operator $\Lambda: X \times \mathbb{C} \to X \times \mathbb{C}$ by 
\[
\Lambda(x, y) := (Tx, y).
\]
Clearly, $\Lambda$ is U-frequently supercyclic, but $\Lambda \times \Lambda$ is not.

\subsection*{Acknowledgment}
The first was partially supported by CAPES. The second author was partially supported by CNPq- Edital Universal Proc. 407854/2021-5

\bibliographystyle{abbrv}
\bibliography{Bibl.bib}
	
\end{document}